\documentclass[12pt]{amsart}
\usepackage{amscd}
\usepackage{amsmath}
\usepackage{amssymb}
\usepackage{units}
\usepackage[all]{xy}
\def\frk{\frak}               

\def\Phi{{\frk n}}
\def\Phi{{\frk N}}
\def\opn#1#2{\def#1{\operatorname{#2}}} 
\opn\chara{char} \opn\length{\ell} \opn\pd{pd} \opn\rk{rk}
\opn\projdim{proj\,dim} \opn\injdim{inj\,dim} \opn\rank{rank}
\opn\depth{depth} \opn\sdepth{sdepth} \opn\fdepth{fdepth}
\opn\grade{grade} \opn\height{height} \opn\embdim{emb\,dim}
\opn\codim{codim}  \opn\min{min} \opn\max{max}

\opn\Tr{Tr} \opn\bigrank{big\,rank}
\opn\superheight{superheight}\opn\lcm{lcm}
\opn\trdeg{tr\,deg}
\opn\reg{reg} \opn\lreg{lreg} \opn\ini{in} \opn\lpd{lpd}
\opn\size{size}
\opn\div{div} \opn\Div{Div} \opn\cl{cl} \opn\Cl{Cl}
\opn\Spec{Spec} \opn\Supp{Supp} \opn\supp{supp} \opn\Sing{Sing}
\opn\Ass{Ass} \opn\Min{Min}
\opn\Ann{Ann} \opn\Rad{Rad} \opn\Soc{Soc}
\opn\Im{Im} \opn\Ker{Ker} \opn\Coker{Coker} \opn\Am{Am}
\opn\Hom{Hom} \opn\Tor{Tor} \opn\Ext{Ext} \opn\End{End}
\opn\Aut{Aut} \opn\id{id}  \opn\deg{deg}

\opn\nat{nat}
\opn\pff{pf}
\opn\Pf{Pf} \opn\GL{GL} \opn\SL{SL} \opn\mod{mod} \opn\ord{ord}
\opn\Gin{Gin} \opn\Hilb{Hilb}
\opn\aff{aff} \opn\con{conv} \opn\relint{relint} \opn\st{st}
\opn\lk{lk} \opn\cn{cn} \opn\core{core} \opn\vol{vol}
\opn\link{link} \opn\star{star}
\opn\gr{gr}

\def\pot#1#2{#1[\kern-0.28ex[#2]\kern-0.28ex]}

\opn\dirlim{\underrightarrow{\lim}}
\opn\inivlim{\underleftarrow{\lim}}
\let\iso=\cong

\let\to=\rightarrow

\def\Implies{\ifmmode\Longrightarrow \else
        \unskip${}\Longrightarrow{}$\ignorespaces\fi}
\def\implies{\ifmmode\Rightarrow \else
        \unskip${}\Rightarrow{}$\ignorespaces\fi}
\def\iff{\ifmmode\Longleftrightarrow \else
        \unskip${}\Longleftrightarrow{}$\ignorespaces\fi}

\let\:=\colon
\newtheorem{Theorem}{Theorem}[section]
\newtheorem{Lemma}[Theorem]{Lemma}
\newtheorem{Corollary}[Theorem]{Corollary}
\newtheorem{Proposition}[Theorem]{Proposition}
\newtheorem{Remark}[Theorem]{Remark}

\newtheorem{Example}[Theorem]{Example}

\let\epsilon\varepsilon
\let\phi=\varphi
\let\kappa=\varkappa
\def\qed{\ifhmode\textqed\fi
      \ifmmode\ifinner\quad\qedsymbol\else\dispqed\fi\fi}
\def\textqed{\unskip\nobreak\penalty50
       \hskip2em\hbox{}\nobreak\hfil\qedsymbol
       \parfillskip=0pt \finalhyphendemerits=0}
\def\dispqed{\rlap{\qquad\qedsymbol}}

\opn\dis{dis}
\def\pnt{{\raise0.5mm\hbox{\large\bf.}}}

\opn\Lex{Lex}

\begin{document}
\title{\bf Depth of  factors of square free  monomial  ideals}

\author{ Dorin Popescu }

\thanks{The  support from  grant ID-PCE-2011-1023 of Romanian Ministry of Education, Research and Innovation is gratefully acknowledged.}

\address{Dorin Popescu,  "Simion Stoilow" Institute of Mathematics of Romanian Academy, Research unit 5,
 P.O.Box 1-764, Bucharest 014700, Romania}
\email{dorin.popescu@imar.ro}

\maketitle
\begin{abstract}
Let $I$ be an ideal of a polynomial algebra  over a field, generated by $r$ square free monomials of degree $d$. If $r$ is bigger  (or equal,  if $I$ is not principal) than the  number of square free monomials of $I$ of degree $d+1$, then $\depth_SI= d$. Let $J\subsetneq I$, $J\not =0$ be generated by square free  monomials of degree $\geq d+1$. If $r$
is bigger
than the number of square free monomials of $I\setminus J$ of degree $d+1$, or more generally the Stanley depth of $I/J$ is $d$, then  $\depth_SI/J= d$. In particular,  Stanley's Conjecture holds in theses cases.
  \vskip 0.4 true cm
 \noindent
  {\it Key words } : Monomial Ideals,  Depth, Stanley depth.\\
 {\it 2000 Mathematics Subject Classification: Primary 13C15, Secondary 13F20, 13F55,
13P10.}
\end{abstract}

\section*{Introduction}

Let $S=K[x_1,\ldots,x_n]$ be the polynomial algebra in $n$ variables over a field $K$,  $d$  a positive integer   and $I\supsetneq J$, be two  square free monomial ideals of $S$ such that $I$ is generated in degrees $\geq d$, respectively $J$ in degrees $d+1$. Let $\rho_d(I)$ be the number of all square free monomials of degree $d$ of $I$. It is easy to note (see Lemma \ref{d}) that
$\depth_SI/J\geq d$. Our Theorem \ref{main} gives a sufficient condition which implies $\depth_SI/J=d$, namely this happens when
$$ \rho_d(I)>\rho_{d+1}(I)-\rho_{d+1}(J).$$
 Suppose that this condition holds. Then the Stanley depth of $I/J$ (see \cite{S}, \cite{HVZ}, or here Remark \ref{st1}) is $d$ and if  Stanley's Conjecture holds then   $\depth_SI/J\leq d$, that is the missing inequality. Thus  to test  Stanley's Conjecture means to test the equality $\depth_SI/J=d$, which is much easier since there exist very good algorithms to compute $\depth_SI/J$ but not so good to compute the Stanley depth of $I/J$. After a lot of examples computed with the computer algebra system SINGULAR we understood that a result as Theorem \ref{main} is believable.  The above condition is not necessary as  Example \ref{gen} shows.   Necessary and sufficient conditions could be possible found  classifying some posets (see Remark \ref{new}) but this is not the subject of this paper.

 The proof of Theorem \ref{main} uses the Koszul homology and can be read without other preparation. Our first section gives easy proofs in special cases, but they are mainly an  introduction in the subject.  Remarks \ref{rem1}, \ref{rem2} show that the Koszul homology seems to be the best tool in our problems. Section $2$ starts with one example, where we give the idea of the proof of Theorem \ref{main}.

  If $I$ is generated by more (or equal, if $I$ is not principal) square free monomials of degree $d$ than ${n\choose d+1}$, or more general than $\rho_{d+1}(I)$, then $\depth_SI=d$ as shows our Corollary \ref{str}. This extends \cite[Corollary 3]{P3}, which was the starting point of our research, the proof there being  easier. Remark \ref{im} says that the condition of Corollary  \ref{str} is tight.

  The conditions above are  consequences of the fact that $\sdepth I/J=d$, as we explained in Remark \ref{st1}, and we saw that they imply $\depth I/J=d$. But what happens if we just suppose that $\sdepth I/J=d$? Then there exists a monomial square free ideal $I'\subset I$  such that $\rho_d(I')>\rho_{d+1}(I')-\rho_{d+1}(I'\cap J)$ using our Theorem \ref{nice} (somehow an extension of  \cite[Lemma 3.3]{Sh}) and  it  follows also $\depth_SI/J=d$ by our Theorem \ref{main2}.

 We owe  thanks to a Referee who found  gaps  in a preliminary version of our paper.

\section{Factors of square free monomial ideals}

Let $J\subsetneq I$,  be  two nonzero square free monomial ideals of $S$ and $d$ a positive integer. Let  $\rho_d(I)$ be the number of all square free monomials of degree $d$ of $I$. Suppose that  $I$ is generated by square free monomials $f_1,\ldots, f_r$, $r>0$, of degrees $\geq d$
and  $J$ is generated by square free monomials of degree $\geq d+1$. Set $s:=\rho_{d+1}(I)-\rho_{d+1}(J)$ and let $b_1,\ldots,b_s$ be the square free monomials of $I\setminus J$ of degree $d+1$.
\begin{Lemma}\label{d}  $\depth_S I,\ \depth_S I/J\geq d$.
\end{Lemma}

\begin{proof} By \cite[Proposition 3.1]{HVZ} we have $\depth_S I\geq d$, $\depth_S J\geq d+1$. The conclusion follows from applying the Depth Lemma in the  exact sequence
$0\to J\to I\to I/J\to 0$.
\end{proof}

\begin{Lemma} \label{r} Suppose that $J=E+F$, $F\not \subset E$, where $E,F$ are  ideals generated by square free monomials of degree $d+1$, respectively $>d+1$. Then $\depth_SI/J=d$ if and only if $\depth_S I/E=d$.
\end{Lemma}

\begin{proof} We may suppose that in $E$ there exist  no monomial generator of $F$.
In the exact sequence
$$0\to J/E\to I/E\to I/J\to 0$$
we see that the first end is isomorphic with $F/(F\cap E)$ and has depth $\geq d+2$ by Lemma \ref{d}. Apply the Depth Lemma and we are done.
\end{proof}

Before trying to extend the above lemma it is useful to see the next example.
\begin{Example} {\em Let $n=4$, $d=1$,  $I=(x_2)$, $E=(x_2x_4)$, $F=(x_1x_2x_3)$. Then $\depth_SI/E=3$ and $\depth_SI/(E+F)=2$.}
\end{Example}

\begin{Lemma} \label{r'} Let $H$ be an ideal  generated by square free monomials of degrees $d+1$. Then $\depth_SI/J=d$ if and only if $\depth_S (I+H)/J=d$.
\end{Lemma}
\begin{proof}
By induction on the number of the generators of $H$ it is enough to consider the case $H=(u)$ for some square free monomial $u\not \in I$ of degrees $d+1$. In the exact sequence $$0\to I/J\to (I+(u))/J\to (I+(u))/I\to 0$$ we see that  the last term is isomorphic with $(u)/I\cap (u)$ and has depth $\geq d+1$ by Lemma \ref{d},  since $I\cap (u) $ has only monomials of degrees $>d+1$. Using the Depth Lemma the first term has depth $d$ if and only if the middle has depth $d$, which is enough.
\end{proof}
Using Lemmas \ref{r}, \ref{r'} we may suppose  always in our consideration that $I$, $J$ are generated in degree $d$, respectively $d+1$, in particular $f_i$   have degree $d$.
\begin{Lemma}\label{s+1} Let $e\leq r$  be a positive integer and  $I'=(f_1,\ldots,f_{e})$, $J'=J\cap I'$. If $\depth_S I'/J'=d$ then $\depth_S I/J=d$ .
\end{Lemma}
\begin{proof} In the exact sequence
$$0\to I'/J'\to I/J\to I/(I'+J)\to 0$$
the right end has depth $\geq d$ by  Lemma  \ref{d} because $$I/(I'+J)\cong (f_{e+1},\ldots,f_r)/((J+I')\cap (f_{e+1},\ldots,f_r)))$$ and
$(J+I')\cap (f_{e+1},\ldots,f_r)$ is generated by monomials of degree $>d$. If the left end has depth $d$ then the middle has the same depth by the Depth Lemma.
\end{proof}

\begin{Lemma}\label{eq} Suppose that there exists $i\in [r]$ such that $f_i$ has in $J$ all square free multiples of degree $d+1$. Then $\depth_S I/J=d$.
\end{Lemma}
\begin{proof} We may suppose $i=1$.  By our hypothesis  $J:f_1$ is generated by
$(n-d)$ variables. If $r=1$ then the depth of $I/J\iso S/(J:f_1)$ is $d$.
If $r>1$ apply the above lemma for $e=1$.
\end{proof}

\begin{Remark}\label{rem1}{\em Suppose in the proof of the above lemma that $f_1=x_1\cdots x_d$. Then the hypothesis says that $(x_{d+1},\ldots,x_n) f_1\subset J$. It follows that $z=f_1 e_{\sigma_1}$, $ e_{\sigma_1}=e_{d+1}\wedge \ldots \wedge e_n$  induces a nonzero element in the Koszul homology module $H_{n-d}(x;I/J)$ of $I/J$ (some details from Koszul homology theory are given in Example \ref{e2}). Thus $\depth_S I/J\leq d$ by \cite[Theorem 1.6.17]{BH}, the other inequality follows from Lemma \ref{d}. This gives a different proof of the above lemma using stronger tools, which will be very useful in the next section. We also remind that $H_{n-d}(x;I/J)\cong \Tor_{n-d}^S(K,I/J)\not =(0)$ gives  $\pd_SI/J\geq n-2$,  which means $\depth_SI/J\leq 2$ by Auslander-Buchsbaum Theorem \cite[Theorem 1.3.3]{BH}.}
\end{Remark}
\begin{Lemma}\label{g} Suppose that $r\geq 2$ and the least common multiple $b=[f_1,f_2]$ has degree $d+1$ and it is the only monomial of degree $d+1$ which is in $(f_1,f_2)\setminus J$. Then $\depth_S I/J=d$.
\end{Lemma}
\begin{proof} Apply induction on $r\geq 2$. Suppose that $r=2$. By hypothesis the greatest common divisor $u=(f_1,f_2)$ have degree $d-1$ and after renumbering the variables we may suppose that $f_i=x_iu$ for $i=1,2$. By hypothesis the square free multiples of $f_1,f_2$ by variables $x_i$, $i>2$ belongs to $J$. Thus we  see that $I/J$ is a finite module over a polynomial ring in $(d+1)$-variables and we get $\depth_S I/J\leq d$ since $I/J$ it is not free. Now it is enough to apply Lemma \ref{d}.
If $r>2$ then apply Lemma \ref{s+1} for $e=2$.
\end{proof}
\begin{Remark}\label{rem2} {\em  We see in the proof of the above lemma (similarly as in Remark \ref{rem1}) that if $u=x_{n-d+2}\cdots x_n$
then $z=f_1 e_{\sigma_1}-f_2 e_{\sigma_2}$, $ e_{\sigma_1}=e_2\wedge \ldots \wedge e_{n-d+1}$, $ e_{\sigma_2}=e_1\wedge e_{3}\wedge \ldots \wedge e_{n-d+1}$ induces a nonzero element in $H_{n-d}(x;I/J)$. Thus $\depth_S I/J\leq d$ again by \cite[Theorem 1.6.17]{BH}.}
\end{Remark}
\begin{Proposition}\label{p} Let $b_1,\ldots,b_s$ be the monomials of degree $d+1$ from $I\setminus J$. Suppose that $r>s$ and for each $i\in [r]$ there exists at most one $j\in [s]$ with $f_i|b_j$. Then $\depth_S I/J=d$.
\end{Proposition}
\begin{proof} If there exists $i\in [r]$ such that $f_i$ has in $J$ all square free multiples of degree $d+1$, then we apply Lemma  \ref{eq}. Otherwise, each $f_i$ has a square free multiple of degree $d+1$ which is not in $J$.   By hypothesis, there exist $i,j\in [r]$, $i\not =j$ such that $f_i,f_j$ have the same multiple $b$ of degree $d+1$ in $I\setminus J$. Now apply the above lemma.
\end{proof}
\begin{Corollary}\label{1}
 Suppose that $r>s\leq 1$.  Then $\depth_S I/J=d$.
\end{Corollary}

\begin{Proposition}\label{2}  Suppose that $r>s=2$.  Then $\depth_S I/J=d$.
\end{Proposition}
\begin{proof} Using Lemma \ref{s+1} for $e=3$ we reduce to the case $r=3$. By Lemma  \ref{eq} we may  suppose that each $f_i$ divides $b_1$, or $b_2$. By Proposition \ref{p}  we may suppose that $f_1|b_1$,  $f_1|b_2$, that is $f_1$ is the greatest common divisor $(b_1,b_2)$. Assume that $f_2|b_1$.
If $f_2|b_2$ then we get $f_2=(b_1,b_2)=f_1$, which is false. Similarly, if $f_3|b_1$  then $f_3\not |b_2$ and we may apply Lemma \ref{g}  to $f_2,f_3$. Thus we reduce to the case when  $f_3|b_2$ and $f_3\not |b_1$. We may suppose that $b_1=x_1f_1$, $b_2=x_2f_1$ and $x_1,x_2$ do not divide $f_1$ because $b_i$ are square free. It follows that $b_1=x_if_2$, $b_2=x_jf_3$ for some $i,j>2$ with $x_i,x_j|f_1$.

{\bf Case $i=j$}

Then we may suppose $i=j=3$ and $f_1=x_3u$ for a square free monomial $u$ of degree $d-1$. It follows that $f_2=x_2u$, $f_3=x_1u$. Let $S'$ be the polynomial subring
of $S$ in the variables $x_1,x_2,x_3$ and those dividing $u$. Then for each variable $x_k\not \in S'$ we have $f_ix_k\in J$ and so $I/J\cong I'/J'$, where $I'=I\cap S'$, $J'=J\cap S'$. Changing from $I,J,S$ to $I',J',S'$  we may suppose that $n=d+2$ and $u=\Pi_{i>3}^n x_i$.
Then $I/J\cong (I:u)/(J:u)\cong (x_1,x_2,x_3)S/(x_1x_2)S$.  Then $\depth_S I/J=d-1+\depth_T (x_1,x_2,x_3)T/(x_1x_2)$. By Lemma \ref{r} it is enough to see that $\depth_T (x_1,x_2,x_3)T/(x_1x_2)T=1$.

{\bf Case $i\not=j$}

Then we may suppose $i=3$, $j=4$ and $f_1=x_3x_4v$ for a square free monomial $v$ of degree $d-2$.  It follows that $f_2=x_1f_1/x_3=x_1x_4v$, $f_3=x_2f_1/x_4=x_2x_3v$. Let $S''$ be the polynomial subring
of $S$ in the variables $x_1,x_2,x_3,x_4$ and those dividing $v$.  As above $I/J\cong I''/J''$, where $I''=I\cap S''$, $J''=J\cap S''$. Changing from $I,J,S$ to $I'',J'',S''$  we may suppose that $n=d+2$ and $v=\Pi_{i>4}^n x_i$. Then
$$I/J\cong (I:v)/(J:v)\cong (x_1x_4,x_2x_3,x_3x_4)S/(x_1x_2x_3,x_1x_2x_4)S.$$
   Then $$\depth_S I/J=d-2+\depth_{T'} (x_1x_4,x_2x_3,x_3x_4)T'/(x_1x_2x_3,x_1x_2x_4)T'.$$ By Lemma \ref{r} it is enough to see that $$\depth_{T'} (x_1x_4,x_2x_3,x_3x_4)T'/(x_1x_2x_3,x_1x_2x_4)T'=2.$$
\end{proof}

\begin{Lemma}\label{use}   Suppose that $d=1$, $f_i=x_i$, $i\in [r]$ and $b_j\in S'=K[x_1,\ldots,x_r]$ for all $j\in [s]$.  Then $\depth_S I/J=1$ independently of $r,s$ ($s$ may be greater than $r$).
\end{Lemma}
\begin{proof} Set $I'=I\cap S'$ and $J'=J\cap S'$. Then $\depth_{S'}S'/I'=0$ and  $\depth_{S'}S'/J'>0$  by Lemma \ref{d}. From the following exact sequence
$$0\to I'/J'\to S'/J'\to S'/I'\to 0$$
it follows that $\depth_{S'}I'/J'=1$ by the Depth Lemma. If $r<n$ then note that $(x_{r+1},\ldots,x_n)I\subset J$ and so
$\depth_S I/J=\depth_S (I'S/J'S) -(n-r)=\depth_{S'} I'/J'=1$.
\end{proof}

\begin{Proposition}\label{3}   Suppose that $d=1$ and $r>s$.  Then $\depth_S I/J=1$.
\end{Proposition}
\begin{proof} By Lemma \ref{use} we may suppose that $I=(x_1,\ldots,x_r)$ with $r<n$. Using Lemma \ref{eq} we may suppose that each $x_i$, $i\in [r]$ divides a certain $b_k$. Apply induction on $s$, the case $s\leq 2$ being done in Proposition \ref{2}. Assume that $s>2$.
We may suppose that each $b_k$ is a product of two different $x_i$, $i\in [r]$ because if let us say $b_s$ is just a multiple of one $x_i$, $i\in [r]$, for example $x_r$,  then we may take $I'=(x_1,\ldots, x_{r-1})$, $J'=J\cap I'$ and we get $\depth_SI'/J'=1$ by induction hypothesis on $s$ since $r-1>s-1$, that is $\depth_S I/J=1$ by Lemma \ref{s+1}. But if  each $b_k$ is a product of two different $x_i$, $i\in [r]$ we see that $b_j\in S'=K[x_1,\ldots,x_r]$ for all $j\in [s]$ and we may apply again Lemma \ref{use}.
\end{proof}

\section{Main result}

We want to extend Proposition \ref{3} for the case $d>1$. Next example is an  illustration of our method.

\begin{Example}\label{e2} {\em Let $n=6$, $d=2$, $f_1=x_1x_6$, $f_2=x_1x_5$, $f_3=x_1x_3$, $f_4=x_3x_4$, \\
$f_5=x_2x_4$,   $$J=(x_1x_2x_4,x_1x_2x_5, x_1x_2x_3,x_1x_2x_6, x_1x_3x_6,x_1x_4x_5, x_1x_4x_6,$$
 $$x_2x_4x_5,x_2x_4x_6,x_3x_4x_5,x_3x_4x_6)$$ and $I=(f_1,f_2,f_3,f_4,f_5)$. We have $s=4$, $b_1=x_5f_1=x_6f_2$, $b_2=x_3f_2=x_5f_3$, $b_3=x_4f_3=x_1f_4$,
 $b_4=x_2f_4=x_3f_5$.
 Let $\partial_i:K_i(x;I/J)\to K_{i-1}(x;I/J)$,  $K_i(x;I/J)\cong S^{{6\choose i}}$, $i\in [6]$ be the Koszul derivation given by
 $$\partial_i(e_{j_1}\wedge \ldots \wedge e_{j_i})=\sum_{k=1}^i (-1)^{k+1}x_{j_k} e_{j_1}\wedge \ldots \wedge e_{j_{k-1}}\wedge e_{j_{k+1}}\wedge\ldots \wedge e_{j_i}.$$
 We consider the following elements of $K_4(x;I/J)$
  $$e_{\sigma_1}=e_2\wedge \ldots \wedge e_5, \ e_{\sigma_2}=e_2\wedge \ldots \wedge e_4\wedge e_6, \ e_{\sigma_3}=e_2\wedge e_4\wedge \ldots \wedge e_6,$$  $$e_{\sigma_4}=e_1\wedge e_2 \wedge e_5\wedge e_6,\ e_{\sigma_5}=e_1\wedge e_3 \wedge e_5\wedge e_6.$$
   Then the element $$z=f_1e_{\sigma_1}-f_2e_{\sigma_2}-f_3e_{\sigma_3}-f_4e_{\sigma_4}+f_5e_{\sigma_5}$$ satisfies
 $$\partial_4(z)=(-b_1+b_1))e_2\wedge e_3\wedge e_4+ (b_2-b_2) e_2\wedge e_4\wedge e_6+$$
 $$(b_3-b_3) e_2\wedge e_5\wedge e_6+ (b_4-b_4)e_1\wedge e_5\wedge e_6=0,$$ since $(b_k)$ are the only monomials of degree $3$ which are not in $J$. Note that in a term $ue_{\sigma}$ of an element from $\Im \partial_5$ we have $u$ of degree $\geq 3$ because $I$ is generated in degree $2$. Thus $z\not \in \Im \partial_5$ induces a nonzero element in $H_4(x;I/J)$. By \cite[Theorem 1.6.17]{BH} we get $\depth_S I/J\leq 2$, which is enough.
}
\end{Example}

\begin{Theorem}\label{main} If $r>s$ then $\depth_S I/J=d$, independently of the characteristic of $K$.
\end{Theorem}
\begin{proof}  Let $\supp f_i=\{j\in [n]: x_j|f_i\}$,\ $e_{\sigma_i}=\wedge_{j\in ([n]\setminus \supp f_i)}\ e_j$ and  \\ $e_{\tau_k}=\wedge_{j\in ([n]\setminus \supp b_k)}\ e_j$. By \cite[Theorem 1.6.17]{BH} it is enough to show, as in the above example, that there exist $y_i\in K$, $i\in [r]$ such that $z=\sum_{i=1}^ry_if_i e_{\sigma_i}$ induces a nonzero element of  $H_{n-d}(x;I/J)$.
Let $\partial_i$ be the Koszul derivation as above. Then $$\partial_{n-d}(z)=\sum_{k=1}^s(\sum_{i\in [r], f_i|b_k} \epsilon_{ki}y_i)b_k $$
 for some  $ \epsilon_{ki}\in \{1,-1\}$. Thus $\partial_{n-d}(z)=0$ if and only if $\sum_{i\in [r], f_i|b_k} \epsilon_{ki}y_i=0$ for all $k\in [s]$. This is a system of  $s$ homogeneous linear equations in $r$ variables $Y$, which must have a nonzero solution in $K$ because $r>s$. As in the above example $z\not\in \Im \partial_{n-d+1}$ if $I$ is generated in degree $ d$ (this may be supposed by Lemmas \ref{r}, \ref{r'}).
\end{proof}

The condition given in Theorem    \ref{main} is tight as shows the following two examples.
\begin{Example} \label{no}{\em Let $n=4$, $d=2$, $f_1=x_1x_3$, $f_2=x_2x_4$, $f_3=x_1x_4$ and $I=(f_1,\ldots,f_3)$, $J=(x_2x_3x_4)$ be ideals of $S$. We have $r=s=3$, $b_1=x_1x_2x_3$, $b_2=x_1x_2x_4$, $b_3=x_1x_3x_4$, and $\depth_SI/J=d+1$.}
\end{Example}

\begin{Example}\label{gen} {\em Let $n=6$, $d=2$, $f_1=x_1x_5$, $f_2=x_2x_3$, $f_3=x_3x_4$, $f_4=x_1x_6$, $f_5=x_1x_4$, $f_6=x_1x_2$, and $I=(f_1,\ldots,f_6)$, $$J=(x_1x_2x_4,x_1x_2x_5,x_1x_3x_5,x_1x_3x_6, x_1x_4x_6, x_2x_3x_5,x_2x_3x_6, x_3x_4x_5,x_3x_4x_6).$$  We have $r=s=6$ and  $b_1=x_1x_4x_5$, $b_2=x_2x_3x_4$, $b_3=x_1x_2x_3$, $b_4=x_1x_5x_6$, $b_5=x_1x_3x_4$, $b_6=x_1x_2x_6$ but $\depth_SI/J=2$ (although $d=2$).}
\end{Example}
\begin{Remark}\label{new}{\em The above example \ref{gen} shows that one could  find a nice class of factors of square free monomial ideals with $r=s$ but $\depth_SI/J=d$ similarly as in \cite[Lemma 6]{P3}. An important tool seems to be a classification of the possible posets given on $f_1,\ldots,f_r,b_1,\ldots,b_s$
 by the divisibility.}
\end{Remark}
\begin{Remark}\label{st1} {\em Given $J\subsetneq I$ two square free monomial ideals of $S$  as above one can consider the poset $P_{I\setminus J}$  of all square free monomials of $I\setminus J$ (a finite set) with the order given by the divisibility. Let ${\mathcal P}$ be a partition of  ${\mathcal P}\ \ P_{I\setminus J}$ in intervals $[u,v]=\{w\in  P_{I\setminus J}: u|w, w|v\}$, let us say   $P_{I\setminus J}=\cup_i [u_i,v_i]$, the union being disjoint.
Define $\sdepth {\mathcal P}=\min_i\deg v_i$ and $\sdepth_SI/J=\max_{\mathcal P} \sdepth {\mathcal P}$, where ${\mathcal P}$ runs in the set of all partitions of $P_{I\setminus J}$. This is the Stanley depth of $I/J$, in fact this is an equivalent definition (see  \cite{S}, \cite{HVZ}).
If $r>s$ then it is obvious that $\sdepth_S I/J=d$ and so
 Theorem \ref{main} says that  Stanley's Conjecture holds, that is $\sdepth_S I/J\geq \depth_S I/J$. ´In general the Stanley depth of a monomial ideal $I$ is greater than or equal with the Lyubeznik' size of $I$ increased by one (see \cite{HPV}).  Stanley's Conjecture holds for intersections of four monomial prime ideals  of $S$ by \cite{AP} and \cite{P1} and for square free monomial ideals of $K[x_1,\ldots,x_5]$ by \cite{P} (a short exposition on this subject  is given in \cite{P2}).
 Also Stanley's Conjecture holds for intersections of three monomial primary ideals by \cite{Z}. In the case of a non square free monomial ideal $I$ a useful inequality is $\sdepth I\leq \sdepth \sqrt{I}$ (see \cite[Theorem 2.1]{Is}).
  }
\end{Remark}

\vskip 0.4 cm

\section{Around Theorem \ref{main}}

Let $S' = K[x_1, \ldots , x_{n-1}]$ be a polynomial ring in $n-1$ variables over a field $K$, $S = S'[x_n]$ and $U, V \subset S'$,
$V\subset U$ be two square free monomial ideals. Set $W = (V +x_nU )S$. Actually,
every monomial square free ideal $T$ of $S$ has this form because then $(T : x_n)$ is generated by an ideal
$U \subset S'$ and $T = (V + x_nU )S$ for $V = T \cap S'$.
\begin{Lemma} (\cite{P}) Suppose that $U \not = V$ and  $\depth_{S'} S'/U = \depth_{S'} S'/V = \depth_{S'} U/V$. Then
$\depth_S S/W = \depth_{S'} S'/U$.
\end{Lemma}

\begin{Lemma} \label{bul} Suppose that $U \not = V$ and  $d:=\depth_{S'} S'/U = \depth_{S'} S'/V$. Then $d= \depth_{S'} U/V$ if and only if $d=\depth_S S/W $.
\end{Lemma}
\begin{proof} The necessity follows from the above lemma. For sufficiency note that in the  exact sequence
$$0\to VS\to W\to US/VS\to 0$$ the depth of the left end is $d+2$  and the middle term has depth $d+1$. It follows that $\depth_S US/VS=d+1$ by the Depth Lemma, which is enough.
\end{proof}

Let $I$ be an ideal of $S$ generated by  square free monomials of degree $\geq d$ and $x_nf_1,\ldots,x_nf_r$, $r>0$ be the  square free monomials of
$I\cap (x_n)$ of degree $d$. Set $U=(f_1,\ldots,f_r)$, $V=I\cap S'$.

\begin{Theorem}\label{main1} If $r>\rho_d(U)-\rho_d(U\cap V)$ then $\depth_S S/I=\depth_{S'}(U+V)/V=d-1$.
\end{Theorem}
\begin{proof} By Theorem \ref{main} we have $\depth_{S'}(U+V)/V=\depth_{S'} U/(U\cap V)=d-1$. Using Lemmas \ref{r}, \ref{r'} we get
$$\depth_{S'} (U+V)/V=\depth_{S'} ((I:x_n)\cap S')/(I\cap S')=d-1.$$
If $\depth_{S'}S'/(I\cap S')=\depth_{S'} S'/((I:x_n)\cap S')=d-1$ then $\depth_S S/I=d-1$ by Lemma \ref{bul}. If $\depth_{S'} S'/((I:x_n)\cap S')=d-2$ then   in the exact sequence
$$0\to S/(I:x_n)\xrightarrow{x_n} S/I\to S'/(I\cap S')\to 0$$
the first term has depth $d-1$ and the other two have depth $\geq d-1$ by Lemma \ref{d}. By the Depth Lemma it follows that $\depth_SS/I=d-1$.

It remains to consider the case when at least one of $\depth_{S'} S'/((I:x_n)\cap S')$ and $\depth_{S'}S'/(I\cap S')$ is $\geq d$. Using the Depth Lemma in the exact sequence
$$0\to ((I:x_n)\cap S')/(I\cap S')\to S'/(I\cap S')\to S'/((I:x_n)\cap S')\to 0$$
we see that necessarily the depth of the last term is $\geq d$ and the depth of the middle term is $d-1$. But then the Depth Lemma applied to the previous exact sequence gives $\depth_SS/I=d-1$ too.
\end{proof}

The following corollary extends \cite[Corollary 3]{P3}.
\begin{Corollary} \label{str} Let $I$ be an ideal generated by $\mu(I)>1$ square free monomials of degree $d$. If $\mu(I)\geq \rho_{d+1}(I)$, in particular if
 $\mu(I)\geq {n\choose d+1}$, then $\depth_SI= d$.
\end{Corollary}
\begin{proof} We have $I=(V+x_n(U+V))S$ as above. Renumbering the variables we may suppose that $U,V\not =0$.  Note that $\mu(I)=r+\rho_d(V)$ and $\rho_{d+1}(I)=\rho_{d+1}(V)+\rho_d(U+V)>\rho_d(V)+\rho_d(U)-\rho_d(U\cap V)$. By hypothesis, $\mu(I)\geq \rho_{d+1}(I)$ and so $r>\rho_d(U)-\rho_d(U\cap V)$. Applying Theorem \ref{main1} we get $\depth_S S/I=d-1$, which is enough.
\end{proof}
\begin{Remark} \label{im}{\em Take in  Example \ref{no} $S'=K[x_1,\ldots,x_5]$ and $L=(J+x_5I)S'$. We have $\mu(L)=4<{5\choose 3+1}$, that is the hypothesis of the above corollary
are not fulfilled. This is the reason that $\depth_{S'} L= 4$ by Lemma \ref{bul} since $\depth_SI/J=3$. Thus the condition of the above corollary is tight.}
\end{Remark}

\section{Minimal  Stanley depth}

Let $S=K[x_1,\ldots,x_n]$ be the polynomial algebra in $n$-variables over a field $K$,  $d$  a positive integer   and $J\subsetneq I$, be two  square free monomial ideals of $S$. Let $\rho_d(I)$ be the number of all square free monomials of degree $d$ of $I$. Suppose that $\rho_d(I)>0$ and $I$ is generated in degree $\geq d$. It follows that $\sdepth_S I/J\geq d$.
\begin{Theorem}\label{nice} The following statements are equivalent:
\begin{enumerate}
\item{} $\sdepth_S I/J=d$
 \item{} there exist some square free monomials  of degree $d$ in $I$, which generate an ideal $I'$ such that $\rho_d(I')>\rho_{d+1}(I')-\rho_{d+1}(I'\cap J)$.
     \end{enumerate}
\end{Theorem}
\begin{proof} If $J\not =0$ then $J$ is generated in degree $\geq d$, even we may suppose that $J$ is generated in degree $\geq d+1$ using an easy isomorphism. Let ${\mathcal M}_d(I)$ be the set of all square free monomials  of $I$ of degree $d$ and ${\mathcal B}={\mathcal M}_{d+1}\setminus J$. We consider the bipartite graph $G$ defined by $V(G)= {\mathcal M}_d(I)\cup {\mathcal B}$, an edge of $G$ can have only endpoints $f\in {\mathcal M}_d(I)$ and $b\in {\mathcal B}$ with $f|b$. Given   $f\in {\mathcal M}_d(I)$ let $\Gamma(f)$ be the set of all vertices $b$ adjacent to $f$ and for $A\subset  {\mathcal M}_d(I)$ set  $\Gamma(A)=\cup_{f\in A}  \Gamma(f)$.
By P. Hall's marriage theorem \cite{vW} there is a complete matching from ${\mathcal M}_d(I)$ to ${\mathcal B}$ if and only if $|\Gamma(A)|\geq |A|$ for every subset $A\subset {\mathcal M}_d(I)$. Thus  $\sdepth_S I/J=d$ if and only if there exists no complete matching above and so there exists a subset $A\subset {\mathcal M}_d(I)$ such that $|\Gamma(A)|< |A|$, that is $I'=(A)$ satisfies the second statement.
\end{proof}
For $J=0$ we get the following corollary, which is closed to \cite[Lemma 3.3]{Sh}
\begin{Corollary} The following statements are equivalent:
\begin{enumerate}
\item{} $\sdepth_S I=d$
 \item{} there exist some square free monomials  of degree $d$ in $I$, which generate an ideal $I'$ such that $\rho_d(I')>\rho_{d+1}(I')$.
     \end{enumerate}
\end{Corollary}
\begin{Theorem}\label{main2} If $\sdepth_S I/J=d$ then $\depth_S I/J=d$, that is Stanley's conjecture holds in this case.
\end{Theorem}
\begin{proof} By  Theorem \ref{nice} there exists a monomial square free ideal $I'\subset I$  such that $\rho_d(I')>\rho_{d+1}(I')-\rho_{d+1}(I'\cap J)$. Then $\depth_S I'/(I'\cap J)=2$ by  Theorem \ref{main} (if $J=0$ we apply  Corollary \ref{str}). Now it is enough to apply Lemma  \ref{s+1}.
\end{proof}

\end{document}